\documentclass[12pt]{amsart}
\usepackage{amsmath,amssymb,graphicx,mathrsfs}
\usepackage{fullpage}
\usepackage{hyperref}
\usepackage{tikz}
\usepackage{color}
\newcommand{\Z}{\mathbb{Z}}

\newtheorem{theorem}[equation]{Theorem}
\newtheorem{lemma}[equation]{Lemma}
\newtheorem{proposition}[equation]{Proposition}

\newtheorem{remark}[equation]{\bf Remark}
\newtheorem{definition}[equation]{\bf Definition}

\numberwithin{equation}{section}
\numberwithin{table}{section}

\author{Dorian Goldfeld \and Eric Stade \and Michael Woodbury}
\address{Dept. of Mathematics\\Columbia University\\ 2990 Broadway \\ New York, NY 10027, USA}
\email{goldfeld@columbia.edu} 
\address{Dept. of Mathematics\\
University of Colorado Boulder\\
Boulder, Colorado 80309, USA}
\email{stade@colorado.edu}
\address{Dept. of Mathematics \\Rutgers, The State University of New Jersey\\
110 Frelinghuysen \phantom{u.} Rd\\
Piscataway, NJ 08854-8019, USA}
\email{michael.woodbury@rutgers.edu} 

\thanks{Dorian Goldfeld is partially supported by Simons Travel Grant: MP-TSM-00001990.}

\title{\protect\parbox{\textwidth}{\protect\centering Multiplicativity of Fourier Coefficients\\ of Maass Forms for SL($ \MakeLowercase n, \mathbb Z$)}}

\begin{document}

\dedicatory{Dedicated to Enrico Bombieri on the occasion of his 85th birthday.}

\maketitle

\begin{abstract}
The Fourier coefficients of a Maass  form $\phi$ for SL$(n,\mathbb Z)$ are complex numbers $A_\phi(M)$, where $M=(m_1,m_2,\ldots,m_{n-1})$ and $m_1,m_2,\ldots ,m_{n-1}$ are nonzero integers.  It is well known that coefficients of the form $A_\phi(m_1,1,\ldots,1)$  are eigenvalues of the Hecke algebra and are multiplicative. We prove that the more general Fourier coefficients $A_\phi(m_1,\ldots,m_{n-1})$ are also eigenvalues of the Hecke algebra and satisfy the multiplicativity relations
$$A_\phi\big(m_1m_1',\;m_2m_2', \;\ldots\; m_{n-1}m_{n-1}'\big) = A_\phi\big(m_1,m_2,\ldots,m_{n-1})\cdot A_\phi(m_1',m_2',\ldots,m_{n-1}'\big)$$
provided the products $\prod\limits_{i=1}^{n-1} m_i$ and $\prod\limits_{i=1}^{n-1} m_i'$ are relatively prime to each other. 
\end{abstract}

\section{\large \bf Introduction}

Let $\pi$ be a unitary cuspidal automorphic representation of GL$(n,\mathbb Q)$ for $n\ge 2.$ Associated to $\pi$ we have the Godement-Jacquet L-function \cite{GJ1972} given by
$$L(s,\pi) = \sum_{n=1}^\infty \frac{\lambda_{\pi}(n)}{n^s}$$  
where the coefficients $\lambda_{\pi}(n)\in\mathbb C$. In the special case of the group SL($n,\mathbb Z$) the Godement-Jacquet L-functions can be studied classically in terms of Maass forms on the quotient space
$\text{\rm SL}(n, \mathbb Z)\backslash \mathfrak h^n$
where $$\mathfrak h^n := \text{\rm GL}(n, \mathbb R)/\left(\text{\rm O}(n, \mathbb R)\cdot \mathbb R^\times\right)$$ is a generalization of the classical upper half-plane. In fact $\mathfrak h^2 := 
\left\{\begin{pmatrix} y&x\\0&1\end{pmatrix}\bigg\vert \; y>0, \, x\in\mathbb R \right\}$
is isomorphic to the classical upper half-plane.

\vskip 5pt
For $n\ge 2$, Maass forms are smooth functions $\phi: \mathfrak h^n\to \mathbb C$
which are automorphic for $\text{\rm SL}(n, \mathbb Z)$ with moderate growth and which are joint eigenfunctions of the full ring of invariant differential operators on $\text{\rm GL}(n, \mathbb R)$ as well as joint eigenfunctions of the  Hecke algebra. The Fourier expansion of Maass forms on GL$(n)$ were obtained for the first time by Piatetski-Shapiro \cite{PS1975} and then
 by Shalika \cite{Sh1973}, \cite{Sh1974}   independently.  

\vskip 5pt
A classical version of the Fourier coefficients of Maass forms on SL$(n,\mathbb Z)$ was announced by Jacquet \cite{J1981} at the Tata Institute  1979 conference on Automorphic Forms, Representation Theory and Arithmetic. In his book  Bump \cite{B1984} explicitly worked out Jacquet's classical approach for GL$(3,\mathbb R)$.
The more general case of GL$(n,\mathbb R)$ was first presented in Goldfeld's book \cite{G2006}.

\vskip 5pt
 A Maass form  $\phi$ for SL($n,\mathbb Z$) has a Fourier expansion (see \cite{G2006}, Theorem 9.3.11)
$$\boxed{\phi(g) =  \sum_{\gamma \,\in\, U_{n-1}(\mathbb Z)\backslash SL(n-1, \mathbb Z)}\;
   \sum_{m_1=1}^\infty  \cdots\sum_{m_{n-2}=1}^\infty \;\, \sum_{m_{n-1}\ne 0}\;\frac{A_\phi\big(m_1,  \ldots, m_{n-1}\big)}{\prod\limits_{k=1}^{n-1} \left|m_k\right|^{\frac{k(n-k)}{2}}}
      \cdot W\Big(M\left(\begin{smallmatrix} \gamma&\\&1\end{smallmatrix}\right)\cdot g\Big)}
     $$
     where $g\in\mathfrak h^n,$
     $M=\left(\begin{smallmatrix}
      m_1\cdots m_{n-2}|m_{n-1}| & & & \\
     &  \hskip -45pt\ddots & & \\ 
         & &\hskip -10pt m_1 & \\  & & &1\end{smallmatrix}\right)$, and $W:\mathfrak h^n\to\mathbb C$  
 is a Whittaker function.    Associated to $\phi$ we have arithmetic Fourier coefficients 
$$A_\phi(m_1,m_2,\ldots,m_{n-1}) \in\mathbb C,$$
where $m_1,m_2,\ldots,m_{n-2}\in \mathbb Z_{\ge1}$ while $m_{n-1}$ is a nonzero integer.   
\vskip 5pt

It is shown in  \cite[ Proposition 9.2.6]{G2006} that every Maass form is either even or odd according to whether 
$A_\phi(m_1,\ldots,m_{n-1}) = \pm A_\phi(m_1,\ldots, -m_{n-1}).$ We assume $\phi$ is normalized so that $A_\phi(1,\ldots,1)=1.$
\vskip 5pt
It is further  shown in   \cite[ Theorem 9.3.11]{G2006} that for each positive integer $m$ there is a Hecke operator $T_m$ acting on the complex vector space of Maass forms of SL$(n,\mathbb Z)$ where
$$T_m \phi(g) = A_\phi(m,1,\ldots,1)\cdot\phi(g), \qquad (g\in \mathfrak h^n).$$
for every Maass form $\phi$. Furthermore $T_{mm'} = T_m T_{m'}$ if $m$ and $m'$ are coprime, i.e. the Hecke operators are multiplicative. 

\begin{definition}{\bf (The Hecke algebras $\mathcal H, \mathcal H_p$ and sets of eigenvalues $\mathcal H^*$, $\mathcal H_p^*$)}
\label{HeckeAlgebra} Fix an integer $n\ge 2$. Let $\mathcal H$ denote the (integral) Hecke algebra which is the commutative polynomial ring over $\Z$ generated by the Hecke operators $T_1,T_2,T_3,\ldots$ acting on automorphic forms for $\mathrm{SL}(n,\mathbb Z)$, i.e., 
$$\mathcal H := \mathbb Z\big[T_1,T_2,T_3,\ldots\big].$$
 We also define $\mathcal H^*$ to be the set of eigenvalues of the Hecke operators $\mathcal H$ acting on $\phi$ where $A_\phi(m,1,\ldots,1)\in \mathcal{H}^*$ for every positive integer $m$.  For a fixed prime $p$  define $\mathcal{H}_p$ to be the subalgebra of $\mathcal{H}$ generated by the Hecke operators $T_{p^k}$  with $k\geq 0$.  Let $\mathcal{H}_p^*$ denote the set of eigenvalues of the Hecke operators in $\mathcal H_p.$
\end{definition}

\vskip 5pt
  Curiously, the Godement-Jacquet L-function $L(s,\phi)$ associated to a Hecke-Maass form $\phi$ is only built up with the eigenvalues of the Hecke operators $A_\phi(m,1,\ldots,1)$ with positive integers $m$ and is defined as
  $$L(s,\phi) := \sum_{m=1}^\infty \frac{A_\phi(m,1,\ldots,1)}{m^s}.$$
Remarkably $L(s,\phi)$
 has an Euler product given by (see \cite{G2006}, Definition 9.4.3)
 
 \small{$$\begin{aligned} & L(s, \phi) = \prod_p \Bigg(1 - \frac{A_{\phi}(p,1,\ldots,1)}{p^{s}}
   + \frac{A_{\phi}(1,p,1,\ldots,1)}{p^{2s}}
  \;\cdots  \; + \;(-1)^{n-1} \frac{A_{\phi}(1,,\ldots,1,p)}{p^{(n-1)s}} +\frac{(-1)^n}{p^{ns}}\Bigg)^{-1}.\end{aligned}$$}
 
\vskip 5pt
The main aim of this paper is to show that the general Fourier coefficients $A_\phi(m_1,m_2,\ldots,m_{n-1})$ are all eigenvalues of elements in the  Hecke algebra and  satisfy the multiplicativity relations
$$\boxed{\phantom{\int} A_\phi\big(m_1m_1',\;m_2m_2', \;\ldots\; m_{n-1}m_{n-1}'\big) = A_\phi\big(m_1,m_2,\ldots,m_{n-1})\cdot A_\phi(m_1',m_2',\ldots,m_{n-1}'\big) \phantom{\int}}$$
provided the products $\prod_{i=1}^{n-1} m_i$ and $\prod_{i=1}^{n-1} m_i'$ are relatively prime to each other. This multiplicativity result is stated in (cf.  \cite[ Theorem 9.3.11]{G2006}) but there is no proof given. Although this is a very well known result to experts, we were unable to find a proof anywhere else in the literature, so this paper fills a possible gap.

\vskip 5pt
\noindent
\underline{\large \bf The main results of this paper:}
 The proof that  $A_\phi(m_1,\ldots,m_{n-1}) \in \mathcal H^*$ is given in (\ref{r-Induction}). The fact that the Fourier coefficients  $A_\phi(m_1,\ldots,m_{n-1})$ are multiplicative is   in the proof of Theorem \ref{MainTheorem}. In sections 3,4 we present some explicit examples of constructing Hecke operators whose eigenvalues are not of the form $A_\phi(m,1,\ldots,1).$ We also remark that all these results can  be proved for Eisenstein series and  residues of Eisenstein series for SL$(n,\mathbb Z)$ with proofs that are essentially the same
as the ones we give for Maass forms.
\section{\large\bf Proof of Multiplicativity}

\begin{definition} Fix an integer $n\ge 2$ and a prime $p$. For $M=(m_1,m_2,\ldots,m_{n-1})\in \mathbb Z^{n-1}$, let  $A_\phi(M)$  denote the $M^{th}$ Fourier coefficient of a Maass form for $\mathrm{SL}(n,\mathbb Z).$  For a positive integer $1\le r\le n-1$ let $K_0,K_1,K_2,\ldots, K_r\in\mathbb Z_{\ge 0}$; assume that $K_0\ge K_1+K_2+\cdots+K_r.$ We define
$$\boxed{\phantom{\Big|} \mathcal A_p(K_0,K_1,\ldots,K_r) := A_\phi\big(p^{K_0},1,\ldots,1)\cdot A_\phi(p^{K_1},\;p^{K_2},\; \ldots, p^{K_r},\; 1,\ldots,1\big).}$$
\end{definition}

\vskip 10pt
We begin with the following lemma which is a key idea in the proof
 of multiplicativity of the Fourier coefficients $A_\phi(M)$.
\begin{lemma} \label{FirstLemma}  Let $p$ be a  prime. Fix  integers $n\ge 2$, $1\le r\le n-1$,  and $K_0,K_1,K_2,\ldots, K_r\in\mathbb Z_{\ge 0}$, with $K_0\ge K_1+K_2+\cdots+K_r$. Then
\small{\begin{align*}\label{eq:lemma}
\mathcal A_p(K_0,K_1,\ldots,K_r)  = \sum_{k_1=0}^{K_1}  \cdots  \sum_{k_r=0}^{K_r} A_\phi\Big(p^{L}, p^{K_2+k_1-k_2},p^{K_3+k_2-k_3},\;\ldots\;, p^{K_{r}+k_{r-1}-k_r},\; p^{k_r},1,\ldots,1  \Big),
\end{align*}}
where
$L = K_0+K_1-2k_1-k_2\;\,\cdots\; -k_r.$
\end{lemma}

\begin{proof}
The proof of Lemma (\ref{FirstLemma}) is based on the following identity (cf.  \cite[p. 277]{G2006}). 

\small{\begin{equation}\label{A-formula}
\boxed{A_\phi(m,1,\ldots,1)A_\phi(m_1,m_2,\ldots,m_{n-1}) = \sum_{\substack{c_1c_2\cdots c_{n} = m \\ c_i\mid m_i\ (1\leq i\leq n-1)}} A_\phi\left( \frac{m_1c_{n}}{c_1},\frac{m_2c_1}{c_2},\frac{m_3c_2}{c_3},\ldots,\frac{m_{n-1}c_{n-2}}{c_{n-1}} \right).}
\end{equation}}
It follows from \eqref{A-formula} that 
 \begin{align*}
\mathcal{A}_p(K_0,K_1,\ldots,K_r) & = \sum_{\substack{(c_1c_2\cdots c_r) c_{n} = p^{K_0} \\ c_i\mid p^{K_i}\ (1\leq i\leq r)}} A_\phi\left( \frac{p^{K_1}c_{n}}{c_1},\frac{p^{K_2}c_1}{c_2},\cdots,\frac{p^{K_r}c_{r-1}}{c_r},c_r,1,\ldots,1 \right) \\
& = \sum_{k_1=0}^{K_1}  \cdots  \sum_{k_r=0}^{K_r} A_\phi\Big(p^{L}, p^{K_2+k_1-k_2},p^{K_3+k_2-k_3},\;\ldots\;, p^{K_{r}+k_{r-1}-k_r},\; p^{k_r},1,\ldots,1  \Big).
\end{align*}
We use here that the sum over $c_i$'s in the first line above is equivalent to the sum over $k_i$'s in the second since $c_i\mid p^{K_i}$ implies we can take $c_i=p^{k_i}$ with $0\leq k_i\leq K_i$.  Also, $c_n = p^{K_0}/p^{k_1+k_2+\cdots+k_r}$, which multiplied by $\frac{m_1}{c_1}=p^{K_1-k_1}$ gives $p^L$ with $L$ as claimed.  
\end{proof}

\vskip 10pt
Fix a prime $p$. Next we prove that every Fourier coefficient of a Maass form $\phi$ for $\mathrm{SL}(n,\mathbb Z)$ of the form $A\big(p^{K_1}, p^{K_2},\ldots,p^{K_{n-1}}\big)$ is an eigenfunction of an element in the Hecke algebra $\mathcal H_p$ as defined in Definition \ref{HeckeAlgebra}. 
\vskip 10pt
\begin{proposition} \label{HeckeEigenvalue} Let $K_1,K_2,\ldots,K_{n-1}\in\mathbb Z_{\ge0}.$ Then $A_\phi\big(p^{K_1}, p^{K_2},\ldots,p^{K_{n-1}}\big) \in \mathcal H_p^*.$
\end{proposition}

\begin{proof} We shall prove  that for every integer $1\le r \le n-1$  
\begin{equation}\label{r-Induction}
A_\phi\big(p^{K_1}, \ldots,p^{K_r},\,1,\ldots,1\big) \in \mathcal H_p^*.
\end{equation} It is obvious that (\ref{r-Induction})  holds for $r=1$ since $A_\phi(p^{K_1},1,\ldots,1)$ is an eigenvalue of $T_{p^{K_1}}.$ We complete the proof by using induction on $r$. Assume (\ref{r-Induction}) holds for every $r\le \frak r$ with $\frak r\ge 1.$ Then we want to prove that (\ref{r-Induction}) holds for $r=\frak r+1.$

Now it follows from Lemma (\ref{FirstLemma}) that,   if $K_0\ge K_1+K_2\cdots+K_r$, then

\small{\begin{align*}
\boxed{\mathcal A_p(K_0,K_1,\ldots,K_\frak r)  = \sum_{k_1=0}^{K_1}  \cdots  \sum_{k_\frak r=0}^{K_\frak r} A_\phi\Big(p^{L}, p^{K_2+k_1-k_2},p^{K_3+k_2-k_3},\;\ldots\;, p^{K_{\frak r}+k_{\frak r-1}-k_\frak r},\; p^{k_\frak r},1,\ldots,1  \Big)}
\end{align*}}
where
$L = K_0+K_1-2k_1-k_2\;\,\cdots\; -k_\frak r.$ It is clear by induction that $\mathcal A_p(K_0,K_1,\ldots,K_\frak r) \in\mathcal H_p^*.$ Note that if we make the simple change change of variables
$$K_0 \to K_0+1,\qquad\quad K_1\to K_1-1$$
then $A_\phi\Big(p^{L}, p^{K_2+k_1-k_2},p^{K_3+k_2-k_3},\;\ldots\;, p^{K_{\frak r}+k_{\frak r-1}-k_\frak r},\; p^{k_\frak r},1,\ldots,1  \Big)$ doesn't change at all. It follows that
\vskip-10pt

\begin{align}& \mathcal A_p(K_0,K_1,\ldots,K_\frak r) - \mathcal A_p(K_0+1,K_1-1,\ldots,K_\frak r)\label{FirstStep}\\
&
\hskip 10pt
= \sum_{k_2=0}^{K_2}  \cdots  \sum_{k_r=0}^{K_\frak r} A_\phi\Big(p^{K_0-K_1-k_2- \;\cdots\;  -k_\frak r}, p^{K_2+K_1-k_2},p^{K_3+k_2-k_3},\;\ldots\;, p^{K_{\frak r}+k_{\frak r-1}-k_\frak r},\; p^{k_\frak r},1,\ldots,1  \Big). \nonumber
\end{align}
Note that the entry $p^{k_\frak r}$ occurs in the ${(\frak r +1)}^{st}$ position.

\vskip 5pt
Next, it is clear that if we make the simple changes
$$K_0\to K_0+1, \qquad K_1\to K_1+1, \qquad K_2\to K_2-1$$
then $A_\phi\Big(p^{K_0-K_1-k_2- \;\cdots\;  -k_\frak r}, p^{K_2+K_1-k_2},p^{K_3+k_2-k_3},\;\ldots\;, p^{K_{\frak r}+k_{\frak r-1}-k_\frak r},\; p^{k_\frak r},1,\ldots,1  \Big)$ doesn't change.

\vskip 10pt
 It follows as before that
\begin{align*}
& \Big(\mathcal A_p(K_0,K_1,K_2, \ldots,K_\frak r) - \mathcal A_p(K_0+1,K_1-1,K_2,\ldots,K_\frak r)\Big)\\
&
\hskip 90pt -  \Big(\mathcal A_p(K_0+1,K_1+1,K_2-1,\ldots,K_\frak r) - \mathcal A_p(K_0+2,K_1,K_2-1,\ldots,K_\frak r)\Big)\\
&
\hskip 46pt
= \sum_{k_3=0}^{K_3}  \cdots  \sum_{k_r=0}^{K_\frak r} A_\phi\Big(p^{K_0-K_1-K_2-k_3- \;\cdots\;  -k_\frak r},\; p^{K_1},  p^{K_3+K_2-k_3},\;\ldots\;, p^{K_{\frak r}+k_{\frak r-1}-k_\frak r},\; p^{k_\frak r},1,\ldots,1  \Big).
\end{align*}
This process can be continued inductively until we reach the final conclusion that
$$A_\phi\Big(p^{K_0-K_1-K_2-K_3- \;\cdots\;  -K_\frak r},\; p^{K_1},  p^{K_2}, p^{K_3},\;\ldots\;,
 p^{K_{\frak r-1}},\; p^{K_\frak r},1,\ldots,1  \Big)\in \mathcal H_p^*.$$
Here we simply choose $K_0$ sufficiently large which allows us to prove this result for every non-negative power of $p$ in the first position of the Fourier coefficient $A_\phi$ above.
\end{proof}

\vskip 10pt
It remains to prove our main theorem. 

\begin{theorem}{\bf (Multiplicativity of Fourier Coefficients)}\label{MainTheorem}
 Let $n\ge 2$ and let $\phi$ be a Maass form for $\mathrm{SL}(n,\mathbb Z)$ with arithmetic Fourier coefficients $A_\phi(m_1,\ldots,m_{n-1})$, normalized so that $$A_\phi(1,1,\ldots,1)=1.$$ Then 
$$\boxed{\phantom{\int} A_\phi\big(m_1m_1',\;m_2m_2', \;\ldots\; m_{n-1}m_{n-1}'\big) = A_\phi\big(m_1,m_2,\ldots,m_{n-1})\cdot A_\phi(m_1',m_2',\ldots,m_{n-1}'\big) \phantom{\int}}$$
if ${\rm gcd}\left(\prod\limits_{i=1}^{n-1} m_i, \,\prod\limits_{i=1}^{n-1}m_i'\right) = 1.$
\end{theorem}

\begin{proof} The proof is based on a simple variant of Lemma (\ref{FirstLemma}) which we now present.

\begin{lemma}\label{SecondLemma} Let $p$ be a fixed prime where $p\nmid m_1m_2\cdots m_{n-1}.$ We have
\begin{align*}
& A_\phi\big(p^{K_0},1,\ldots,1\big)\cdot A_\phi\Big(p^{K_1}m_1,\;p^{K_2}m_2,\; \ldots, p^{K_r}m_r,\; m_{r+1},\ldots,m_{n-1}\Big)\\
&
\hskip 23pt
=  \sum_{k_1=0}^{K_1}  \cdots  \sum_{k_r=0}^{K_r} A_\phi\Big(p^{L}m_1,\; p^{K_2+k_1-k_2}m_2,\ldots\;, p^{K_{r}+k_{r-1}-k_r}m_r,\; p^{k_r}m_{r+1},m_{r+2}, \;\ldots,m_{n-1}  \Big)
\end{align*}
where $L = K_0+K_1-2k_1-k_2\;\,\cdots\; -k_r.$
\end{lemma}
\begin{proof} The proof is exactly the same as the proof of Lemma (\ref{FirstLemma}). In fact, since $p\nmid m_1m_2\cdots m_{n-1}$, it follows that the sum  in equation (\ref{A-formula}) which takes the form
$$ \sum_{\substack{c_1c_2\cdots c_{n} = p^{K_0} \\ c_i\mid m_i\ (1\leq i\leq n-1)}} $$
 tells us that each $c_i = p^{k_i}$ as before since all the $c_i$ have to divide $p^{K_0}$.
The proof immediately follows from equation (\ref{A-formula}).
\end{proof}

\vskip 10pt

\noindent
{\bf Completion of the proof of Theorem  \ref{MainTheorem}.}
 It is enough to prove that  if $p$ is a fixed prime where $p\nmid m_1m_2\cdots m_{n-1}$ then for every every $1\le r\le n-1$ and all $K_1,\ldots,K_r\in\mathbb Z_{\ge0}$ we have
 
\begin{equation}\label{mult}
A_\phi\left(p^{K_1}m_1,\,\ldots ,p^{K_r}m_r,\, m_{r+1},\cdots,m_{n-1}  \right) = A_\phi\left(p^{K_1},\,\ldots ,p^{K_r},\, 1,\cdots,1 \right) \cdot A_\phi\left(m_1,\,\ldots ,m_{n-1} \right).
\end{equation}

\vskip 5pt
 It follows easily from Lemma \ref{SecondLemma} that (\ref{mult}) holds for $r=1.$ We complete the proof of  (\ref{mult}) by using induction on $r$. Assume (\ref{mult}) holds for every $r\le \frak r$ with $\frak r\ge 1.$ Then we want to prove that (\ref{mult}) holds for $r=\frak r+1.$ Now it follows from Lemma \ref{SecondLemma} that

\begin{align}\label{NewIdentity1}
& A_\phi\big(p^{K_0},1,\ldots,1\big)\cdot A_\phi\Big(p^{K_1}m_1,\;p^{K_2}m_2,\; \ldots, p^{K_\frak r}m_\frak r,\; m_{\frak r+1},\ldots,m_{n-1}\Big)\\
&
\hskip 23pt
=  \sum_{k_1=0}^{K_1}  \cdots  \sum_{k_r=0}^{K_r} A_\phi\Big(p^{L}m_1,\; p^{K_2+k_1-k_2}m_2,\ldots\;, p^{K_{\frak r}+k_{\frak r-1}-k_\frak r}m_\frak r,\; p^{k_\frak r}m_{\frak r+1},m_{\frak r+2}, \;\ldots,m_{n-1}  \Big)\nonumber\end{align}
where $L = K_0+K_1-2k_1-k_2\;\,\cdots\; -k_r.$
If we make the change $K_0\to K_0+1, \; K_1\to K_1-1$ then the coefficient $A_\phi(*)$ on the right hand side of (\ref{NewIdentity1}) does not change at all. It follows that

\begin{align}\label{NewIdentity2}
& \Bigg(A_\phi\big(p^{K_0},1,\ldots,1\big)\cdot A_\phi\Big(p^{K_1}m_1,\;p^{K_2}m_2,\; \ldots, p^{K_\frak r}m_\frak r,\; m_{\frak r+1},\ldots,m_{n-1}\Big)\Bigg)\\
& 
\hskip 60pt - \Bigg(A_\phi\big(p^{K_0+1},1,\ldots,1\big)\cdot A_\phi\Big(p^{K_1-1}m_1,\;p^{K_2}m_2,\; \ldots, p^{K_\frak r}m_\frak r,\; m_{\frak r+1},\ldots,m_{n-1}\Big)\Bigg)\nonumber\\
&
\hskip 40pt
=  \sum_{k_2=0}^{K_2}  \cdots  \sum_{k_r=0}^{K_r} A_\phi\Big(p^{L}m_1,\; p^{K_2+K_1-k_2}m_2,\ldots\;, p^{K_{\frak r}+k_{\frak r-1}-k_\frak r}m_\frak r,\; p^{k_\frak r}m_{\frak r+1},m_{\frak r+2}, \;\ldots,m_{n-1}  \Big)\nonumber\end{align}
where $L = K_0-K_1-k_2\;\,\cdots\; -k_r.$
\vskip 5pt

Now note, that by the inductive hypothesis, the left side of (\ref{NewIdentity2}) can be written as
\begin{align*}
& \Bigg(A_\phi\big(p^{K_0},1,\ldots,1\big)\cdot A_\phi\Big(p^{K_1},\;p^{K_2},\; \ldots, p^{K_\frak r},\; 1,\ldots,1\Big)\\ 
&
\hskip 50pt- A_\phi\big(p^{K_0},1,\ldots,1\big)\cdot A_\phi\Big(p^{K_1},\;p^{K_2},\; \ldots, p^{K_\frak r},\; 1,\ldots,1\Big)\Bigg)\cdot A_\phi\big(m_1,\,\ldots ,m_{n-1} \big)
\end{align*}
The induction process can be continued in exactly the same way as the proof of Proposition \ref{HeckeEigenvalue} leading to the final result that

\begin{align*}
&A_\phi\Big(p^{L},\; p^{K_1},  p^{K_2}, p^{K_3},\;\ldots\;,
 p^{K_{\frak r-1}},\; p^{K_\frak r},1,\ldots,1  \Big)\cdot A_\phi\big(m_1,\,\ldots ,m_{n-1} \big) \\
 &
 \hskip 120pt
 = A_\phi\Big(p^{L}m_1,\;p^{K_1}m_2,\; \ldots, p^{K_{\frak r-1}}m_\frak r,\;  p^{K_{\frak r}}m_{\frak r+1},\ldots,m_{n-1}\Big)
 \end{align*}
 with $L= K_0-K_1-K_2-K_3- \;\cdots\;  -K_\frak r.$
\end{proof}

\begin{remark} If we assume all $m_i=1$ (for $i=1,\ldots,n-1$), the first inductive step (\ref{NewIdentity2}) in the above proof is exactly the same as the first inductive step (\ref{FirstStep}) in the proof of 
Proposition \ref{HeckeEigenvalue}. In fact all the inductive steps in the proof of 
the multiplicativity relation (\ref{mult}) will exactly match the inductive steps in the proof of Proposition \ref{HeckeEigenvalue} if all $m_i=1.$ When the $m_i$ are not all equal to 1 (since $p\nmid m_i$ for $i=1,\ldots,n-1$) there is really no change in the Hecke identity \ref{A-formula} except that
the $m_i$ are inserted in the  $i^{\rm th}$ place of the Fourier coefficient $A_\phi$. 
\end{remark}

\section{\large \bf The Example \boldmath $A_\phi(1,\ldots,1,p,1,\ldots,1)$}

 Recall that a {\it composition} of a positive integer $\ell$ is an ordered tuple $(i_1,i_2,\ldots, i_r)$ ($r\in\mathbb Z_{\ge1}$) of positive integers such that $$i_1+i_2+\cdots+i_r=\ell.$$
 
 We have:
 
 \begin{proposition}\label{ellth-place}For $n\ge 2$ and $ 1\le \ell\le n-1$, let $\mathcal C_\ell$ denote the set of compositions of $\ell$.  Then
\begin{equation}
A_\phi(1,\ldots,1,\hskip-6pt\underbrace{p}_{\ell^{\rm th}{\rm\, place}}\hskip-5pt,1,\ldots,1)=\sum_{(i_1,i_2,\ldots, i_r)\in C_\ell} \,\prod_{j=1}^r (-1)^{i_j+1}A_\phi( p^{i_j},1,\ldots,1).\label{ellth-formula}
\end{equation}
In other words, the  combination of Hecke operators whose eigenvalue equals the left hand side of \eqref{ellth-formula} is $\sum\limits_{(i_1,i_2,\ldots, i_r)\in C_\ell} \,\prod\limits_{j=1}^r (-1)^{i_j+1}T_{p^{i_j}}$.
   \end{proposition}

To prove this proposition, we need two lemmas, for which, in turn, we introduce some notation.

\begin{definition}\rm Let $n\ge2$; let $p$ be prime; let $j\in \mathbb Z_{\ge0}$ and $0\le k\le n$.  We define\begin{align*} A_{j,k}(p)=\begin{cases} A_\phi(p^j,1,\ldots,1 )&\text{ if $k=0$ or $k=n$;}\\
 A_\phi(p^{j+1},1,\ldots,1 )&\text{ $k=1$;}\\
 A_\phi(p^j,1,\ldots,1,\hskip-6pt\underbrace{p}_{k^{\rm th}{\rm\, place}}\hskip-5pt,1,\ldots,1)&\text{ if $2\le k\le n-1$.}\end{cases}\end{align*}
 \end{definition}
 That is,  $A_{j,k}(p)$ is the result of multiplying the $k$th coordinate in the argument of  $A_\phi(p^j,1,\ldots,1 )$ by $p$,  where the cases $k=0$ and $k=n$ correspond to no extra factor of $p$.  
 
We have
\begin{lemma}\label{steplemma}  For $n\ge2$, $p$ prime, $j\in\mathbb Z_{\ge1},$ and $1\le k\le n-1$, 
\begin{equation} A_{j,0}(p)A_{0,k}(p)= A_{j,k}(p)+A_{j-1,k+1}(p).\label{jk-recur}\end{equation}
 \end{lemma}

\begin{proof} Into \eqref{A-formula}, we put$$ m=p^j,\quad m_k=p,\quad    \quad m_i=1\  (1\le i\le n-1,\ i\ne k),$$so that the  left hand side of \eqref{A-formula} equals $A_{j,0}(p)A_{0,k}(p) $. Then the conditions$$c_1c_2\cdots c_n=m;\quad  c_i|m_i\ (1\le i\le n-1)$$imply that the sum in \eqref{A-formula} entails two summands:  either when $c_n=p^j$ and  $c_i=1$ for $1\le i\le n-1$; or when $c_n=p^{j-1}$, $c_k=p$,  and $c_i=1$ for $1\le i\le n-1$ and $i\ne k$.  In the first case, the corresponding summand equals  $A_{j,k}(p)$; in the second case, this summand equals $A_{j-1,k+1}(p)$. From this, \eqref{jk-recur} follows. \end{proof}

A consequence of the above lemma is the following.

\begin{lemma} \label{ell-recur}For $2\le \ell\le n-1$, we have
\begin{equation}A_\phi(1,\ldots,1,\hskip-6pt\underbrace{p}_{\ell^{\rm th}{\rm\, place}}\hskip-5pt,1,\ldots,1)=\sum_{m=1}^\ell (-1)^{m+1} A_\phi(p^m,1,\ldots,1)
A_\phi(1,\ldots,1,\hskip-16pt\underbrace{p}_{(\ell-m)^{\rm th}{\;\rm place}}\hskip-15pt,1,\ldots,1),\label{recur}
\end{equation}
with the understanding that, when $\ell=m$,
$$A_\phi(1,\ldots,1,\hskip-16pt\underbrace{p}_{(\ell-m)^{\rm th}{\;\rm place}}\hskip-15pt,1,\ldots,1)$$simply  denotes the constant 1.
\end{lemma}

\begin{proof}  Putting $j=1$ and $k=\ell-1$ into  \eqref{jk-recur}, and rearranging, gives
\begin{equation} A_{0,\ell}(p) =A_{1,0}(p)A_{0,\ell-1}(p)-A_{1,\ell-1}(p).\label{one-pass}\end{equation}Next,  to the term  $A_{1,\ell-1}(p)$ in \eqref{one-pass}, we apply \eqref{jk-recur} with $j=2$ and $k=\ell-2$.  We get 
\begin{equation*} A_{0,\ell}(p) =A_{1,0}(p)A_{0,\ell-1}(p)-A_{2,0}(p)A_{0,\ell-2}(p)+A_{2,\ell-2}(p).\end{equation*}Iterating this process ultimately yields \begin{align*} A_{0,\ell}(p) &=A_{1,0}(p)A_{0,\ell-1}(p)-A_{2,0}(p)A_{0,\ell-2}(p)+A_{3,0}(p)A_{0,\ell-3}(p)-\cdots+(-1)^{\ell+1} A_{\ell,0}(p),\end{align*}which is precisely the statement \eqref{recur}. \end{proof}
 
\begin{proof}[Proof of Proposition \ref{ellth-place}]  We apply strong induction on $\ell$ (for fixed $n$).  
 
 The formula \eqref{ellth-formula} is clearly true in the case $\ell=1$.  So assume that it holds for any of the integers $1,2,\ldots, \ell-1$ in place of $\ell$.  Then, by Lemma 1.6, we have 

 \begin{align}
 A_\phi(1,\ldots,1,\hskip-6pt\underbrace{p}_{\ell^{\rm th}{\rm\, place}}\hskip-5pt,1,\ldots,1)&=\sum_{m=1}^\ell (-1)^{m+1} A_\phi(p^m,1,\ldots,1)\cdot\hskip-20pt
\sum_{(i_1,i_2,\ldots, i_r)\in C_{\ell-m}} \,\prod_{j=1}^r (-1)^{i_j+1}A_\phi( p^{i_j},1,\ldots,1)\nonumber\\&=
\sum_{m=1}^\ell \sum_{(i_1,i_2,\ldots, i_r)\in C_{\ell-m}}\hskip-20pt (-1)^{m+1} A_\phi(p^m,1,\ldots,1) 
 \,\prod_{j=1}^r (-1)^{i_j+1}A_\phi( p^{i_j},1,\ldots,1).\label{recur2}
 \end{align}
 
 But, for $1\le m\le \ell$, $(i_1,i_2,\ldots, i_r)$ is a composition of $\ell-m$ if and only if $(i_1,i_2,\ldots, i_r,m)$ is a composition of $\ell$.  So, putting $m=i_{r+1}$ into \eqref{recur2}, we get
  \begin{align*}A_\phi(1,\ldots,1,\hskip-6pt\underbrace{p}_{\ell^{\rm th}{\rm\, place}}\hskip-5pt,1,\ldots,1)&=
\sum_{(i_1,i_2,\ldots,i_r, i_{r+1})\in C_{\ell }}
 \,\prod_{j=1}^{r+1}(-1)^{i_j+1}A_\phi( p^{i_j},1,\ldots,1). \qedhere\end{align*}
\end{proof}
 
 \section{\large\bf The Example \boldmath $A_\phi(1,p^j,1,\ldots,1)$}
 
 We have:
 
 \begin{proposition}\label{p2-first-result} {\it For $n\ge3$ and $j\in \mathbb Z_{\ge1}$, we have
\begin{equation} A_\phi(1,p^j,1,\ldots,1)=A_\phi(p^j,1,\ldots,1)^2-A_\phi(p^{j-1},1,\ldots,1)A_\phi(p^{j+1},1,\ldots,1). \label{p2-second-place}\end{equation}
In other words, the  combination of Hecke operators whose eigenvalue is $A_\phi(1,p^j,1,\ldots,1)$ is $(T_{p^j})^2-T_{p^{j-1}}T_{p^{j+1}}$.}
\end{proposition}

\begin{proof}  Let $n\ge3$ and $j\in \mathbb Z_{\ge1}$.  

We first put  $$m=m_1=p^j,\quad m_i=1\ (2\le i\le n-1)$$into \eqref{A-formula}.  For such $m_i$'s and $m$, the set of $n$-tuples $(c_1,c_2,\ldots, c_n)$ such that $c_1c_2\cdots c_n=m$ and $c_i|m_i$ for $1\le i\le n-1$ is just the set
$$\{(p^k,1,\ldots,1,p^{j-k})\in( \mathbb Z_{\ge1})^n\,|\,0\le k\le j\}.$$ 
So \eqref{A-formula} yields
\begin{align}A_\phi(p^j,1,\ldots,1)^2&=\sum_{k=0}^j A\biggl(\frac{p^j\cdot p^{j-k}}{p^k},\frac{1\cdot p^k}{1},\frac{1\cdot1}{1},\ldots,\frac{1\cdot1}{1}\biggr) \nonumber\\&=\sum_{k=0}^j A_\phi(p^{2j-2k},p^k,1,\ldots,1).\label{pjsquared}\end{align}
Next:  into \eqref{A-formula}, we put  $$m= p^{j-1},\quad m_1=p^{j+1},\quad m_i=1\ (2\le i\le n-1).$$In this case,  the set of $n$-tuples $(c_1,c_2,\ldots, c_n)$ such that $c_1c_2\cdots c_n=m$ and $c_i|m_i$ for $1\le i\le n-1$ equals the set$$\{(p^k,1,\ldots,1,p^{j-1-k})\in (\mathbb Z_{\ge1})^n\,|\,0\le k\le j-1\}.$$So \eqref{A-formula} gives
\begin{align}A_\phi(p^{j-1},1,\ldots,1)A_\phi(p^{j+1},1,\ldots,1) &=\sum_{k=0}^{j-1} A\biggl(\frac{p^{j+1}\cdot p^{j-1-k}}{p^k},\frac{1\cdot p^k}{1},\frac{1\cdot1}{1},\ldots,\frac{1\cdot1}{1}\biggr)\nonumber\\&=\sum_{k=0}^{j-1}A_\phi(p^{2j-2k},p^k,1,\ldots,1).\label{pjpm1}\end{align}Subtracting \eqref{pjpm1} from \eqref{pjsquared} yields
$$A_\phi(p^j,1,\ldots,1)^2-A_\phi(p^{j-1},1,\ldots,1)A_\phi(p^{j+1},1,\ldots,1)=A_\phi(p^{2j-2j},p^j,1,\ldots,1)=A_\phi(1,p^j,1,\ldots,1),$$which is the desired result.
 \end{proof}


\begin{thebibliography}{99} 
   
  \bibitem{B1984} 
Bump, D.,   Automorphic Forms on $GL(3, \Bbb R)$, Lecture Notes in Math. {\bf 1083}, Springer-Verlag (1984).
 \vskip 8pt
 
\bibitem{GJ1972}
Godement, R. and Jacquet, H.,
{\it Zeta functions of simple algebras,} Lecture notes in Math. {\bf  260}, Springer-Verlag (1972).
\vskip 8pt

\bibitem{G2006}
 Goldfeld, D., Automorphic forms and L-functions for the group
${\rm GL}(n,\mathbf R)$. With an appendix by Kevin A. Broughan. Cambridge Studies in Advanced Mathematics,
  vol. 99. Cambridge: Cambridge University Press (2006).
   \vskip 8pt
   
   
\bibitem{J1981}  
    Jacquet,H., {\it Dirichlet series for the group $GL(N)$,} in: Automorphic Forms,  Representation Theory and Arithmetic, Tata Inst. of Fund. Res., Springer-Verlag  (1981), 155--164.
 \vskip 8pt
      
\bibitem{PS1975}   
 Piatetski-Shapiro,I.I., 
 {\it Euler subgroups.  Lie groups and their representations,} Proc. Summer School, Bolyai Jnos Math. Soc., Budapest, 1971, Halsted, New York  (1975), 597--620. 
\vskip 8pt
   
 \bibitem{Sh1973}
  Shalika,J.A., {\it  On the multiplicity of the spectrum of the space of cusp forms of $GL_n$,}  Bull. Amer. Math. Soc.  {\bf 79}  (1973), 454--461. 
 \vskip 8pt
 
 \bibitem{Sh1974}
   Shalika, J.A.,  {\it The multiplicity one theorem for GL(n),} Annals of Math. {\bf 100} ((1974), 171--193.
 


 \end{thebibliography}
   \end{document}